\documentclass[12pt, reqno]{amsart} 
\usepackage{amsmath, amsthm, amscd, amsfonts, amssymb, graphicx, color}
\usepackage[bookmarksnumbered, colorlinks, plainpages]{hyperref}
\hypersetup{colorlinks=true,linkcolor=red, anchorcolor=green, citecolor=cyan, urlcolor=red, filecolor=magenta, pdftoolbar=true}

\newtheorem{theorem}{Theorem}[section]
\newtheorem{lemma}[theorem]{Lemma}

\newtheorem{cor}[theorem]{Corollary}
\theoremstyle{definition}

\newtheorem{example}[theorem]{Example}

\theoremstyle{remark}
\newtheorem{remark}[theorem]{\bf{Remark}}
\numberwithin{equation}{section}
\begin{document}

	\title [Power numerical radius inequalities]    
	{{ Power numerical radius inequalities from an extension of Buzano's inequality }}
	
	\author[P. Bhunia]{Pintu Bhunia}

	\address{ {Department of Mathematics, Indian Institute of Science, Bengaluru 560012, Karnataka, India}}
	\email{pintubhunia5206@gmail.com; pintubhunia@iisc.ac.in}

	%\thanks will become a 1st page footnote.
	%\thanks{We would like to thank the referee for his/her extremely fruitful suggestions. First author  would like to thank UGC, Govt. of India for the financial support in the form of SRF. Prof. Kallol Paul would like to thank RUSA 2.0, Jadavpur University for the partial support.}
	
	\thanks{Dr. Pintu Bhunia  would like to thank SERB, Govt. of India for the financial support in the form of National Post Doctoral Fellowship (N-PDF, File No. PDF/2022/000325) under the mentorship of Professor Apoorva Khare}
	
	\thanks{}
	%    Information for second author
	%    General info
	
	
	\date{\today}
	\maketitle

\begin{abstract}
	Several numerical radius inequalities are studied by developing an extension of the Buzano’s inequality. It is shown  that if $T$ is a bounded linear operator on a complex Hilbert space, then
	\begin{eqnarray*}
	w^n(T) &\leq& \frac{1}{2^{n-1}} w(T^n)+ \sum_{k=1}^{n-1} \frac{1}{2^{k}} \left\|T^k \right\| \left\|T \right\|^{n-k},
\end{eqnarray*}
for every positive integer $n\geq 2.$ This is a non-trivial improvement of the classical inequality $w(T)\leq \|T\|.$ The above inequality gives an estimation for the numerical radius of the nilpotent operators, i.e., if $T^n=0$ for some least positive integer $n\geq 2$, then
\begin{eqnarray*}
	w(T) &\leq& \left(\sum_{k=1}^{n-1} \frac{1}{2^{k}} \left\|T^k \right\| \left\|T \right\|^{n-k}\right)^{1/n}
	\leq \left( 1- \frac{1}{2^{n-1}}\right)^{1/n} \|T\|.
\end{eqnarray*}
  Also, we deduce a reverse inequality for the numerical radius power inequality $w(T^n)\leq w^n(T)$. We show that
 if $\|T\|\leq 1$, then 
 \begin{eqnarray*}
 	w^n(T) &\leq& \frac{1}{2^{n-1}} w(T^n)+  1- \frac{1}{2^{n-1}},
 \end{eqnarray*}
 for every positive integer $n\geq 2.$ This inequality is sharp.
\end{abstract}

%where $w(T)$ and $\|T\|$ are the numerical radius and the usual operator norm of $T,$ respectively.

\section{\textbf{Introduction}}

\noindent Let $\mathcal{B}(\mathcal{H})$ denote the $C^*$-algebra of all bounded linear operators on a complex Hilbert space $\mathcal{H}$ with usual inner product $\langle.,. \rangle$ and the corresponding norm $\|\cdot\|.$
Let $T\in \mathcal{B}(\mathcal{H})$ and let $|T|=(T^*T)^{1/2}$, where $T^*$ denotes the adjoint of $T.$ The numerical radius and the usual operator norm of $T$ are denoted by $w(T)$ and $\|T\|,$ respectively. The numerical radius of $T$ is defined as $$ w(T)=\sup \left\{|\langle Tx,x\rangle| :   x\in \mathcal{H}, \|x\|=1 \right\}.$$
It is well known that the numerical radius defines a norm on $\mathcal{B}(\mathcal{H})$ and is equivalent to the usual operator norm. Moreover, \text{for every $T\in \mathcal{B}(\mathcal{H})$,} the following inequalities hold:
\begin{eqnarray}\label{eqv}
	\frac12 \|T\| \leq w(T) \leq \|T\|.
\end{eqnarray}
The inequalities are sharp, $w(T)=\frac12 \|T\|$ if $T^2=0$ and $w(T)=\|T\|$ if $T$ is normal. Like as the usual operator norm, the numerical radius satisfies the power inequality, i.e., 
\begin{eqnarray}\label{power}
	w(T^n) \leq w^n(T) 
\end{eqnarray}
 \text{for every positive integer $n.$}
 One other basic property for the numerical radius is the weakly unitary invariant property, i.e., $w(T)=w(U^*TU)$ for every unitary operator $U\in \mathcal{B}(\mathcal{H}).$
To study more interesting properties of the  numerical radius we refer to \cite{book,book2}. 

\noindent The numerical radius has various applications in many branches in sciences, more precisely,  perturbation problem, convergence problem, approximation problem and iterative method as well as recently developed quantum information system. Due to importance of the numerical radius,  many eminent mathematicians have been studied numerical radius inequalities to improve the inequalities \eqref{eqv}. Various inner product inequalities play important role to study numerical radius inequalities. 
The Cauchy-Schwarz inequality is one of the most useful inequality, which states that for every \text{$x,y \in \mathcal{H}$,}
\begin{eqnarray}\label{cauchy}
	|\langle x,y\rangle| &\leq& \|x\| \|y\|.
\end{eqnarray} 
A generalization of the Cauchy-Schwarz inequality is the Buzano's inequality \cite{Buzano}, which states that for $x,y,e \in \mathcal{H}$ with $\|e\|=1,$
\begin{eqnarray}\label{buzinq}
	|\langle x,e\rangle \langle e,y\rangle| &\leq& \frac{ |\langle x,y\rangle|  +\|x\| \|y\|}{2}.
\end{eqnarray}
Another generalization of the Cauchy-Schwarz inequality is the mixed Schwarz inequality \cite[pp. 75--76]{Halmos}, which states that for every $x,y\in \mathcal{H}$ and $T\in \mathcal{B}(\mathcal{H}),$
\begin{eqnarray}\label{mixed}
	|\langle Tx,y\rangle|^2 &\leq& \langle |T|x,x\rangle \langle |T^*|y,y\rangle.
\end{eqnarray}
Using the above inner product inequalities various mathematicians have developed various numerical radius inequalities, which improve the inequalities \eqref{eqv},  see \cite{Bhunia_ASM_2022, Bhunia_LAMA_2022,Bhunia_RIM_2021,Bhunia_BSM_2021,Bhunia_ADM_2021,D08,Kittaneh_2003}. Also using other technique various nice numerical radius inequalities have been studied, see \cite{Abu_RMJM_2015,Bag_MIA_2020,Bhunia_LAA_2021,Bhunia_LAA_2019,Kittaneh_LAMA_2023, Kittaneh_STD_2005, Yamazaki}. Haagerup and Harpe \cite{Haagerup} developed a nice estimation for the numerical radius of the nilpotent operators, i.e., if $T^n=0$ for some positive integer $n\geq 2$, then 
\begin{eqnarray}\label{haag}
	w(T) &\leq&   \cos \left(\frac{\pi}{n+1}  \right) \|T\|.
\end{eqnarray}
\smallskip
\noindent 

In this paper, we obtain a generalization of the Buzano's inequality \eqref{buzinq} and using this generalization we develop new numerical radius inequalities, which improve the existing ones. From the numerical radius inequalities obtained here,  we deduce several results. We deduce an estimation for the nilpotent operators like  \eqref{haag}. Further, we deduce a reverse inequality for the numerical radius power inequality \eqref{power}.

\section{\textbf{Numerical radius inequalities}}
We begin our study with proving the following lemma, which is a generalization of the Buzano's inequality \eqref{buzinq}. 

%We use a concept, studied in \cite{Fujii}, to prove this lemma.

\begin{lemma}\label{buz-extension}
	Let $x_1,x_2,\ldots,x_n,e \in \mathcal{H}$, where $\|e\|=1$. Then
	\begin{eqnarray*}
		\left| \mathop{\Pi}\limits_{k=1}^n \langle x_k,e\rangle \right| &\leq& \frac{ \left| \langle x_1,x_2\rangle \mathop{\Pi}\limits_{k=3}^n \langle x_k,e\rangle\right| + \mathop{\Pi}\limits_{k=1}^n \|x_k\|}{2}.
	\end{eqnarray*}

%In particular, for $n=3$
%\begin{eqnarray*}
%			|  \langle x_1,e\rangle \langle x_2,e\rangle \langle x_3,e\rangle|  &\leq & \frac{|  \langle x_1,x_2\rangle\langle x_3,e\rangle|+ \|x_1\|\|x_2\| \|x_3\|}{2}.
%\end{eqnarray*}

\end{lemma}

\begin{proof}
	Following the inequality \eqref{buzinq}, we have
	$$ |\langle x_1, e\rangle \langle x_2,e\rangle| \leq \frac{|\langle x_1, x_2\rangle |+ \|x_1\|\|x_2\|}{2}.$$
	By replacing $x_2$ by $ \mathop{\Pi}\limits_{k=3}^n \langle x_k,e\rangle  x_2$
	and using $\left| \mathop{\Pi}\limits_{k=3}^n \langle x_2,e\rangle\right| \leq \mathop{\Pi}\limits_{k=3}^n \| x_k \|$, we obtain the desired inequality. 
\end{proof}

%..................................................

%\begin{lemma}\label{buz-extension}
%	Let $x,y,z,e\in \mathcal{H}$, where $\|e\|=1.$ Then
%	\begin{eqnarray*}
%		|  \langle x,e\rangle \langle y,e\rangle \langle z,e\rangle|  &\leq & \frac{|  \langle x,y\rangle\langle z,e\rangle|+ \|x\|\|y\| \|z\|}{2}.
%	\end{eqnarray*}
%\end{lemma}

%	\begin{proof}
%		Suppose $P$ is an orthogonal projection. Then $2P-I$ is also an orthogonal projection. Then, from the Cauchy-Schwarz inequality \eqref{cauchy}, we have
%		$$|\langle (2P-I)x \langle z,e \rangle ,y \rangle|\leq \|x\| \|y\| \|z\|.  $$
%		This implies that
%		\begin{eqnarray}\label{eqn1}
%			 2|\langle Px,y \rangle \langle z,e \rangle|-|\langle x,y \rangle \langle z,e \rangle| \leq \|x\| \|y\| \|z\|.
%		\end{eqnarray}
%		Suppose $(x\otimes y)z= \langle z,y \rangle x.$ Then it is easy to verify that $e\otimes e$ is an orthogonal projection, since $\|e\|=1.$ Therefore, considering $P=e\otimes e$ in \eqref{eqn1}, we obtain  
%		$$2| \langle x,e \rangle \langle e,y \rangle \langle z,e \rangle|-|\langle x,y \rangle \langle z,e \rangle| \leq \|x\| \|y\| \|z\|,$$ as desired.
%	\end{proof}

%	If we take $x_3,x_4,\ldots,x_n=e$ in Lemma \ref{buz-extension}, then we get the Buzano's inequality \eqref{buzinq}.
	Now using Lemma \ref{buz-extension} (for $n=3$), we prove the following numerical radius inequality. % which improves the second inequality in \eqref{eqv}. 
	
\begin{theorem}\label{th1}
	Let $T\in \mathcal{B}(\mathcal{H})$. Then 
	\begin{eqnarray*}
		w(T) &\leq& \sqrt[3]{\frac14 w(T^3) + \frac14\left( \|T^2\|+ \|T^*T+TT^*\|\right) \|T\|}.
	\end{eqnarray*}
\end{theorem}
\begin{proof}
	Take $x\in \mathcal{H}$ and $\|x\|=1.$ From Lemma \ref{buz-extension} (for $n=3$), we have
	\begin{eqnarray*}
		|\langle Tx,x\rangle |^3 &=& |\langle Tx,x\rangle\langle T^*x,x\rangle\langle T^*x,x\rangle|\\
		&\leq& \frac{|\langle Tx,T^*x\rangle \langle T^*x,x\rangle|+ \|Tx\| \|T^*x\|^2}{2}\\
		&=& \frac{|\langle T^2x,x\rangle \langle T^*x,x\rangle|}{2}+ \frac{(\|Tx\|^2+ \|T^*x\|^2)\|T^*x\|}{4}
	\end{eqnarray*}
Also, from the Buzano's inequality \eqref{buzinq}, we have
\begin{eqnarray*}
	|\langle T^2x,x\rangle \langle T^*x,x\rangle| &\leq& \frac{|\langle T^2x,T^*x\rangle|+ \|T^2x\| \|T^*x\|   }{2}\\
	 &=& \frac{|\langle T^3x,x\rangle|+ \|T^2x\| \|T^*x\|   }{2}.
\end{eqnarray*}
Therefore, 
	\begin{eqnarray*}
	|\langle Tx,x\rangle |^3 
	&\leq&  \frac{|\langle T^3x,x\rangle|+ \|T^2x\| \|T^*x\|}{4} + \frac{(\|Tx\|^2+ \|T^*x\|^2)\|T^*x\|}{4}\\
	&\leq& \frac14 w(T^3)+ \frac14 \|T^2\| \|T\|+\frac14 \|T^*T+TT^*  \| \|T \|.
	\end{eqnarray*}
Therefore, taking supremum over $\|x\|=1,$ we get the desired inequality.
\end{proof}

The inequality in Theorem \ref{th1} is an improvement of the second inequality in \eqref{eqv}, since $w(T^3)\leq \|T^3\|\leq \|T\|^3$, $\|T^2\|\leq \|T\|^2$ and $\|T^*T+TT^*\|\leq 2\|T\|^2.$ To show non-trivial improvement, we consider a matrix $T=\begin{bmatrix} 
		0&2&0\\
		0&0&1\\
		0&0&0
	\end{bmatrix}.$ Then $w(T^3)=0$ and so 
$$  \sqrt[3]{  \frac14 w(T^3) + \frac14\left( \|T^2\|+ \|T^*T+TT^*\|\right) \|T\| } < \|T\|.$$

Next, using the Buzano's inequality \eqref{buzinq}, we obtain the following numerical radius inequality.

\begin{theorem}\label{th2}
	Let $T\in \mathcal{B}(\mathcal{H}).$ Then 
	\begin{eqnarray*}
		w(T) &\leq& \sqrt[3]{\frac12 w(TT^*T)+\frac14 \|T^*T+TT^*\| \|T\|}.
	\end{eqnarray*}
\end{theorem}
\begin{proof}
	Take $x\in \mathcal{H}$ and $\|x\|=1.$ By the Cauchy-Schwarz inequality \eqref{cauchy}, we have
	\begin{eqnarray*}
		|\langle Tx,x\rangle |^3 &\leq& \|Tx\|^2|\langle T^*x,x\rangle|
		= |\langle |T|^2x,x\rangle \langle T^*x,x\rangle|.
	\end{eqnarray*}
From the Buzano's inequality \eqref{buzinq}, we have
\begin{eqnarray*}
	|\langle Tx,x\rangle |^3 
	&\leq& \frac{|\langle |T|^2x,T^*x\rangle|+ \||T|^2x\| \|T^*x\|}{2}\\
	&\leq& \frac{|\langle T|T|^2x,x\rangle |+ \||T|^2x\| \|T^*x\|}{2}\\
	&\leq& \frac{|\langle T|T|^2x,x\rangle |+ \|T\| \|Tx\| \|T^*x\|}{2}\\
	&\leq& \frac{|\langle T|T|^2x,x\rangle |}{2}+\frac{  (\|Tx\|^2+ \|T^*x\|^2)\|T\|}{4}\\
	&\leq& \frac12 w(T|T|^2)+ \frac14 \|T^*T+TT^*\| \|T\|.
\end{eqnarray*}
Therefore, taking supremum over $\|x\|=1,$ we get the desired inequality.
\end{proof}

Clearly, for every $T\in \mathcal{B}(\mathcal{H})$, $$\sqrt[3]{\frac12 w(TT^*T)+\frac14 \|T^*T+TT^*\| \|T\|}\leq \sqrt[3]{\frac12 w(TT^*T)+\frac12 \|T\|^3} \leq \|T\|.$$

% Replacing $T$ by $T^*$ in Theorem \ref{th2}, we have
% \begin{eqnarray}\label{eqn4}
% w^3(T)&\leq& \frac12 w(|T^*|^2T)+ \frac12 \|T\|^3.
% \end{eqnarray}   

Also, using similar technique as Theorem \ref{th2}, we can prove the following numerical radius inequality.

\begin{eqnarray}\label{eqn5}
	w(T)&\leq& \sqrt[3]{\frac12 w(T^*T^2)+ \frac12 \|T\|^3}.
\end{eqnarray}
And also replacing $T$ by $T^*$ in  \eqref{eqn5}, we get
\begin{eqnarray}\label{eqn6}
	w(T)&\leq& \sqrt[3]{\frac12 w(T^2T^*)+ \frac12 \|T\|^3}.
\end{eqnarray}

Consider a matrix 
$T=\begin{bmatrix} 
	0&2&0\\
	0&0&1\\
	0&0&0
\end{bmatrix}$. Then we see that $$ w(T^2 T^*)=1< w(T^*T^2)=2< w(TT^*T)=\sqrt{65}/2 .$$

Therefore, combining Theorem \ref{th2} and the inequalities \eqref{eqn5} and \eqref{eqn6}, we obtain the following corollary.

\begin{cor}\label{cor2}
	Let $T\in \mathcal{B}(\mathcal{H})$. Then 
\begin{eqnarray*}
	w(T) &\leq& \sqrt[3]{\frac12 \min \Big(w(TT^*T), w(T^2T^*), w(T^*T^2)  \Big) + \frac12  \|T\|^3}.
\end{eqnarray*}
\end{cor}

\smallskip

Since
$w(TT^*T)\leq \|T\|^3, \ w(T^2T^*)\leq \|T^2\| \|T\|\leq \|T\|^3,\ w(T^*T^2)\leq \|T^2\| \|T\|\leq \|T\|^3,$
 the inequality in Corollary \ref{cor2} is an improvement of the second inequality in \eqref{eqv}.

\smallskip

Now, Theorem \ref{th2} together with the inequalities \eqref{eqn5} and \eqref{eqn6} implies the following result.

\begin{cor}\label{coor2}
	Let $T\in \mathcal{B}(\mathcal{H})$. If $w(T)=\|T\|,$ then 
	\begin{eqnarray*}
		w(TT^*T)= w(T^2T^*)= w(T^*T^2)= \|T\|^3.
	\end{eqnarray*}
\end{cor}

Next, by using Lemma \ref{buz-extension} (for $n=3$), the Buzano's inequality \eqref{buzinq} and the mixed Schwarz inequality \eqref{mixed}, we obtain the following result.

\begin{theorem}\label{th3}
	Let $T\in \mathcal{B}(\mathcal{H}).$ Then
	\begin{eqnarray*}
		w(T) &\leq& \sqrt[3]{\frac14 w(|T|T|T^*|)+ \frac14 \Big (  \|T^2\|+ \|T^*T+TT^*\| \Big)\|T\|}.
	\end{eqnarray*}
\end{theorem}

\begin{proof}
	Let $x\in \mathcal{H}$ with $\|x\|=1.$ From the mixed Schwarz inequality \eqref{mixed}, we have
	\begin{eqnarray*}
		|\langle Tx,x\rangle|^3 &\leq & \langle |T^*|x,x\rangle |\langle T^*x,x\rangle| \langle |T|x,x\rangle. 
	\end{eqnarray*}
Using Lemma \ref{buz-extension} (for $n=3$), we have
\begin{eqnarray*}
	|\langle Tx,x\rangle|^3 &\leq & \frac{|\langle |T^*|x,T^*x\rangle  \langle |T|x,x\rangle | + \||T^*|x\| \|T^*x\| \||T|x\| } {2}\\
	&= & \frac{|\langle T |T^*|x,x\rangle | \langle |T|x,x\rangle  + \||T^*|x\| \|T^*x\| \||T|x\| } {2}. 
\end{eqnarray*}
	By Buzano's inequality \eqref{buzinq}, we have
\begin{eqnarray*}
	 |\langle T |T^*|x,x\rangle | \langle |T|x,x\rangle  
	 &\leq& \frac{ | \langle T |T^*|x,|T|x\rangle |+ \|T |T^*|x\| \||T|x\|  }{2}\\
	 &=& \frac{| \langle |T| T |T^*|x, x\rangle |+ \|T |T^*|x\| \||T|x\|  }{2}.\\
	\end{eqnarray*}
Also by AM-GM inequality, we have
\begin{eqnarray*}
	 \||T^*|x\| \|T^*x\| \||T|x\| &\leq& \frac12(\||T^*|x\|^2 + \||T|x\|^2)  \|T^*x\|\\
	 &=& \frac12 \langle (|T|^2+|T^*|^2)x,x\rangle \|T^*x\|.
\end{eqnarray*}
Therefore, 
\begin{eqnarray*}
	|\langle Tx,x\rangle|^3 &\leq & \frac{ |\langle |T| T |T^*|x, x\rangle |+ \|T |T^*|x\| \||T|x\|  }{4}  +  \frac14 \langle (|T|^2+|T^*|^2)x,x\rangle \|T^*x\|\\
	&\leq& \frac14 w(|T|T|T^*|)+ \frac14 \Big (  \|T |T^*|\|+ \|T^*T+TT^*\| \Big)\|T\|. 
\end{eqnarray*}
From the polar decomposition $T^*=U|T^*|$, it is easy to verify that  $\|T |T^*|\|=\|T^2\|.$ So,
\begin{eqnarray*}
	|\langle Tx,x\rangle|^3 
	&\leq& \frac14 w(|T|T|T^*|)+ \frac14 \Big (  \|T^2 \|+ \|T^*T+TT^*\| \Big)\|T\|. 
\end{eqnarray*}
Therefore, taking supremum over $\|x\|=1,$ we get the desired result.
\end{proof}

%Now it is easy verify (using the polar decomposition of $T$ and $T^*$) that 
%$$w(|T|T|T^*|)\leq \| |T|T|T^*| \|= \|T^3\|.$$
%Therefore, for every $T\in \mathcal{B}(\mathcal{H}),$  the following inequality holds.
%\begin{eqnarray*}
%	&& \frac14 w(|T|T|T^*|)+ \frac14 \Big (  \|T^2 \|+ \|T^*T+TT^*\| \Big)\|T\|\\
%&&  \leq \frac14 \|T^3\|+ \frac14 \Big (  \|T^2 \|+ \|T^*T+TT^*\| \Big)\|T\|. 
%\end{eqnarray*}

%Now following \cite[Theorem 2.3]{Bhunia_LAA_2019}, we have
%	\begin{eqnarray}\label{eqn11}
%	w^3(T) &\leq& \frac14 w(T^3)+ \frac14 \Big (  \|T^2\|+ \|T^*T+TT^*\| \Big)\|T\|.
%\end{eqnarray}

%Consider a matrix 
%$T=\begin{bmatrix} 
%	0&2&0\\
%	0&0&1\\
%	0&0&0
%\end{bmatrix}$, then we see that $$ w(|T| T |T^*|)=0< w(|T^*|T|T|)=\sqrt{65}/2 .$$

Now, combining Theorem \ref{th3} and Theorem  \ref{th1}, we obtain the following corollary.

\begin{cor}\label{cor3}
	Let $T\in \mathcal{B}(\mathcal{H}),$ then
	\begin{eqnarray*}
		w(T) \leq \sqrt[3]{\frac14 \min \Big( w(T^3), w(|T|T|T^*|) \Big)+ \frac14 \Big (  \|T^2\|+ \|T^*T+TT^*\| \Big)\|T\|}.
	\end{eqnarray*}
\end{cor}

Using similar technique as Theorem \ref{th3}, we can also prove the following inequality. 
\begin{eqnarray}\label{eqn7}
	w(T) &\leq& \sqrt[3]{\frac14 w(|T^*|T|T|)+ \frac38    \|T^*T+TT^*\| \|T\|}.
\end{eqnarray}

Clearly, the inequalities in Corollary \ref{cor3} and \eqref{eqn7} are stronger than the second inequality in \eqref{eqv}. 

And when $w(T)=\|T\|,$ then 
$$ w(|T|T|T^*|)=  w(|T^*|T|T|) =w(T^3)=\|T\|^3.$$

Next theorem reads as follows:

\begin{theorem}\label{th4}
		Let $T\in \mathcal{B}(\mathcal{H}).$ Then
	\begin{eqnarray*}
		w(T) &\leq&   \sqrt[4]{ \Big( \frac12 w(T|T|) +\frac14  \|T^*T+TT^*\|   \Big)   \Big ( \frac12 w(T^*|T^*|)+ \frac14\|T^*T+TT^*\| \Big)}.
	\end{eqnarray*}
\end{theorem}
\begin{proof}
	Let $x\in \mathcal{H}$ with $\|x\|=1.$ From the mixed Schwarz inequality \eqref{mixed}, we have
	$$ |\langle Tx,x\rangle|^2 \leq \langle |T|x,x\rangle \langle |T^*|x,x\rangle.$$
	Therefore, 	$$ |\langle Tx,x\rangle|^4 \leq \langle |T|x,x\rangle \langle |T^*|x,x\rangle |\langle Tx,x\rangle \langle T^*x,x\rangle|.$$
	From the Buzano's inequality \eqref{buzinq}, we have
	\begin{eqnarray*}
		 |\langle |T|x,x \rangle \langle T^*x,x\rangle |
		  &\leq& \frac{|\langle |T|x, T^*x \rangle|+ \||T|x\| \|T^*x\|}{2}\\
		 &\leq&  \frac12 |\langle T|T|x, x \rangle|+ \frac14  (\||T|x\|^2+ \|T^*x\|^2)\\
		 &\leq& \frac12 w(T|T|)+ \frac14 \|T^*T+TT^*\|.
	\end{eqnarray*}
Similarly, 	\begin{eqnarray*}
	|\langle |T^*|x,x \rangle \langle Tx,x\rangle |
	&\leq& \frac{|\langle |T^*|x, Tx \rangle|+ \||T^*|x\| \|Tx\|}{2}\\
	&\leq&  \frac12 |\langle T^*|T^*|x, x \rangle|+ \frac14  (\||T^*|x\|^2+ \|Tx\|^2)\\
	&\leq& \frac12 w(T^*|T^*|)+ \frac14 \|T^*T+TT^*\|.
\end{eqnarray*}
Therefore, $$|\langle Tx,x\rangle|^4 \leq \left(\frac12 w(T|T|)+ \frac14 \|T^*T+TT^*\|\right) \left( \frac12 w(T^*|T^*|)+ \frac14 \|T^*T+TT^*\|\right).$$
Taking supremum over $\|x\|=1,$ we get
$$w^4(T) \leq \left(\frac12 w(T|T|)+ \frac14 \|T^*T+TT^*\|\right) \left( \frac12 w(T^*|T^*|)+ \frac14 \|T^*T+TT^*\|\right),$$ as desired.
\end{proof}

Again, using similar technique as Theorem \ref{th4}, we can prove the following result.

\begin{theorem}\label{th5}
	Let $T\in \mathcal{B}(\mathcal{H}).$ Then
	\begin{eqnarray*}
		w(T) &\leq&   \sqrt[4]{ \Big( \frac12 w(T|T^*|) +\frac12  \|T\|^2   \Big)   \Big ( \frac12 w(T^*|T|)+ \frac12\|T\|^2 \Big)}.
	\end{eqnarray*}
\end{theorem}

Clearly, the inequalities in Theorem \ref{th4} and Theorem \ref{th5} are improvements of the second inequality in \eqref{eqv}. The inequalities imply that when $w(T)=\|T\|$, then 
$$w(T|T|)=w(T^*|T^*|)  = w(T|T^*|)=  w(T^*|T|)=\|T\|^2.$$

Now we consider the following example to compare the inequalities in Theorem \ref{th4} and Theorem \ref{th5}.
 
\begin{example}
	Consider a matrix $T=\begin{bmatrix}
		0&1&0\\
		0&0&1\\
		0&0&0
	\end{bmatrix},$ then Theorem \ref{th4} gives $w(T)\leq \sqrt{ \frac{1}{2\sqrt{2}}+ \frac12 }$, whereas  Theorem \ref{th5}  gives $w(T)\leq \sqrt{ \frac{1}{4}+ \frac12}.$

Again, Consider $T=\begin{bmatrix}
	0&1&0\\
	0&0&2\\
	0&0&0
\end{bmatrix},$ then Theorem \ref{th4} gives $w(T)\leq \sqrt{ \frac{\sqrt{17}+5}{4} }$, whereas  Theorem \ref{th5}  gives $w(T)\leq \sqrt{ \frac{5+5}{4} }.$ 
Therefore, we would like to note that the inequalities obtained in Theorem \ref{th4} and Theorem \ref{th5} are not comparable, in general.
\end{example}

\section{\textbf{Numerical radius inequalities involving general powers}}

We develop a numerical radius inequality involving general powers  $w^n(T)$ and $w(T^n)$ for every positive integer $n\geq 2$ and from which we derive nice results related to the nilpotent operators and reverse power inequality for the numerical radius. First we prove the following theorem.
 %, which is a generalization of Lemma \ref{buz-extension}. 

%\begin{lemma}\label{lemgen}
%	Let $x_1,x_2,\ldots,x_n,e \in \mathcal{H}$, where $\|e\|=1$. Then
%	\begin{eqnarray*}
%		\left| \mathop{\Pi}\limits_{k=1}^n \langle x_k,e\rangle \right| &\leq& \frac{ \left| \langle x_1,x_2\rangle \mathop{\Pi}\limits_{k=3}^n \langle x_k,e\rangle\right| + \mathop{\Pi}\limits_{k=1}^n \|x_k\|}{2}.
%	\end{eqnarray*}
%\end{lemma}

%\begin{proof}
%	Following Lemma \ref{buz-extension}, we have
%	$$ |\langle x_1, e\rangle \langle x_2, e\rangle \langle x_3, e\rangle| \leq \frac{|\langle x_1, x_2\rangle \langle x_3, e\rangle|+ \|x_1\|\|x_2\| \|x_3\|}{2}.$$
%	By replacing $x_3$ by $ \mathop{\Pi}\limits_{k=4}^n \langle x_k,e\rangle  x_3$
%	 and using $\left| \mathop{\Pi}\limits_{k=4}^n \langle x_k,e\rangle\right| \leq \mathop{\Pi}\limits_{k=4}^n \| x_k \|$, we obtain the desired inequality. 
%\end{proof}

\begin{theorem}\label{th7}
	If $T\in \mathcal{B}(\mathcal{H}),$ then
	\begin{eqnarray*}
		|\langle Tx, x\rangle|^n &\leq& \frac{1}{2^{n-1}} \left|\langle T^nx, x\rangle \right|+ \sum_{k=1}^{n-1} \frac{1}{2^{k}} \left\|T^kx \right\| \left\|T^*x \right\|^{n-k},
	\end{eqnarray*}
for all $x\in \mathcal{H}$ with $\|x\|=1$ and for every positive integer $n\geq 2.$
\end{theorem}
\begin{proof}
	We have
	\begin{eqnarray*}
	&&	|\langle Tx,x \rangle|^n \\
		&=&|\langle Tx,x \rangle \langle T^*x,x \rangle \langle T^*x,x \rangle^{n-2} |\\
		&\leq& \frac{\left|\langle Tx,T^*x \rangle \langle T^*x,x \rangle^{n-2}  \right|+ \|Tx\| \|T^*x\|^{n-1}  }{2} \,\, (\text{by Lemma \ref{buz-extension}})\\
		&=& \frac{\left|\langle T^2x,x \rangle \langle T^*x,x \rangle \langle T^*x,x \rangle^{n-3}  \right|+ \|Tx\| \|T^*x\|^{n-1}  }{2}\\
		&\leq & \frac{ \frac{\left|\langle T^2x,T^*x \rangle  \langle T^*x,x \rangle^{n-3}  \right|+ \|T^2x\| \|T^*x\|^{n-2} }{2}+ \|Tx\| \|T^*x\|^{n-1}  }{2}  \,\, (\text{by Lemma \ref{buz-extension}})\\
		&= & \frac{\left|\langle T^3x,x \rangle \langle T^*x,x \rangle \langle T^*x,x \rangle^{n-4}  \right|+ \|T^2x\| \|T^*x\|^{n-2} }{2^2}+\frac{  \|Tx\| \|T^*x\|^{n-1}  }{2} \\
		&\leq & \frac{\frac{\left|\langle T^3x,T^*x \rangle  \langle T^*x,x \rangle^{n-4}\right|+ \|T^3x\| \|T^*x\|^{n-3} }{2} + \|T^2x\| \|T^*x\|^{n-2} }{2^2} 
		 +\frac{  \|Tx\| \|T^*x\|^{n-1}  }{2} \\
		 && \,\,\,\,\,\,\,\,\,\,\,\,\,\,\,\,\,\,\, \,\,\,\,\,\,\,\,\,\,\,\,\,\,\,\,\,\,\,\,\,\,\,\,\,\,\,\,\,\,\,\,\,\,\,\,\,\,\,\,\,\,\,\,\,\,\,\,\,\,\,\,\,\,\,\,\,\,\,\,\,\,\,\,\,\,\,\,\,\,\,\,\,\,\,\,\,\,\,\,\,\,\,\,\,\,\,\,\,\,\,\,\,\,\,\,\,\,\,\,\,\,\,\,\,\,\,\,\,\,\, (\text{by Lemma \ref{buz-extension}})\\
		 &= &  \frac{\left|\langle T^4x,x \rangle  \langle T^*x,x \rangle  \langle T^*x,x \rangle^{n-5}\right|+ \|T^3x\| \|T^*x\|^{n-3} }{2^3}  + \frac{ \|T^2x\| \|T^*x\|^{n-2} }{2^2} \\
		  && \,\,\,\,\,\,\,\,\,\,\,\,\,\,\,\,\,\,\, \,\,\,\,\,\,\,\,\,\,\,\,\,\,\,\,\,\,\,\,\,\,\,\,\,\,\,\,\,\,\,\,\,\,\,\,\,\,\,\,\,\,\,\,\,\,\,\,\,\,\,\,\,\,\,\,\,\,\,\,\,\,\,\,\,\,\,\,\,\,\,\,\,\,\,\,\,\,\,\,\,\,\,\,\,\,\,\,\,\,\, +\frac{  \|Tx\| \|T^*x\|^{n-1}  }{2}.
		\end{eqnarray*}
	Repeating this approach $(n-1)$ times (i.e., using Lemma \ref{buz-extension}), we obtain
	\begin{eqnarray*}
		|\langle Tx, x\rangle|^n &\leq& \frac{1}{2^{n-1}} \left|\langle T^nx, x\rangle \right|+ \sum_{k=1}^{n-1} \frac{1}{2^{k}} \left\|T^kx \right\| \left\|T^*x \right\|^{n-k},
	\end{eqnarray*}
	as desired.
\end{proof}

The following generalized numerical radius inequality is a simple consequence of Theorem \ref{th7}.

\begin{cor}\label{cor5}
	If $T\in \mathcal{B}(\mathcal{H}),$ then
	\begin{eqnarray}\label{pp}
		w^n(T) &\leq& \frac{1}{2^{n-1}} w(T^n)+ \sum_{k=1}^{n-1} \frac{1}{2^{k}} \left\|T^k \right\| \left\|T \right\|^{n-k},
	\end{eqnarray}
for every positive integer $n\geq 2.$
\end{cor}

\begin{remark}
(i)	For every $T\in \mathcal{B}(\mathcal{H})$ and for every positive integer $n\geq 2$, 
	\begin{eqnarray*}
		w^n(T) &\leq& \frac{1}{2^{n-1}} w(T^n)+ \sum_{k=1}^{n-1} \frac{1}{2^{k}} \left\|T^k \right\| \left\|T \right\|^{n-k} \\
		&\leq& \frac{1}{2^{n-1}} w(T^n)+ \sum_{k=1}^{n-1} \frac{1}{2^{k}}  \left\|T \right\|^{n} \\
		&\leq& \frac{1}{2^{n-1}} \|T^n\|+ \sum_{k=1}^{n-1} \frac{1}{2^{k}}  \left\|T \right\|^{n} \,\, \text{(by \eqref{eqv})}\\
		 &\leq& \frac{1}{2^{n-1}} \|T\|^n+ \sum_{k=1}^{n-1} \frac{1}{2^{k}}  \left\|T \right\|^{n} 
		= \|T\|^n. 
	\end{eqnarray*}
	Therefore, the inequality \eqref{pp} is an improvement of the second inequality in \eqref{eqv}.
	
\noindent (ii) From the above inequalities, it follows that if $w(T)=\|T\|$, then
	$$ w^n(T)=w(T^n)=\|T^n\|=\|T\|^n,$$
for every positive integer $n\geq 2.$ And so it is easy to see that when $w(T)=\|T\|$, then
 $$w(T)= \lim\limits_{n\to \infty} \|T^n\|^{1/n}= r(T),$$ where 
$r(T)$ denotes the spectral radius of $T.$ 
The second equality holds for every operator $T\in \mathcal{B}(\mathcal{H})$ and it is known as Gelfand formula for spectral radius.

\noindent (iii) Taking $n=2$ in Corollary \ref{cor5}, we get
$$ w^2(T) \leq \frac12 w(T^2)+ \frac12\|T\|^2,$$
which was  proved by Dragomir \cite{D08}.

\noindent (iv) Taking $n=2$ in Theorem \ref{th7}, we deduce that
$$ w^2(T) \leq \frac12 w(T^2)+ \frac14 \|T^*T+TT^*\|,$$
which was proved by Abu-Omar and  Kittaneh \cite{Abu_RMJM_2015}.
\end{remark}

Following Corollary \ref{cor5}, we obtain an estimation for the nilpotent operators.

\begin{cor}\label{nilpotent}
	Let $T\in \mathcal{B}(\mathcal{H}).$
	If $T^n=0$ for some least positive integer $n\geq 2$, then
\begin{eqnarray*}
	w(T) &\leq& \left(\sum_{k=1}^{n-1} \frac{1}{2^{k}} \left\|T^k \right\| \left\|T \right\|^{n-k}\right)^{1/n}
	\leq \left( 1- \frac{1}{2^{n-1}}\right)^{1/n} \|T\|.
\end{eqnarray*}
\end{cor}

Consider a matrix $T=\begin{bmatrix} 
	0&1&2\\
	0&0&3\\
	0&0&0
\end{bmatrix}.$ Then we see that Corollary \ref{nilpotent} gives $w(T)\leq \alpha \approx 3.0021$ and Theorem \ref{th1} gives $w(T) \leq \beta \approx 2.5546,$ whereas the inequality \eqref{haag} gives $w(T)\leq \gamma \approx 2.5811.$ Therefore, we conclude that for the nilpotent operators the inequality \eqref{haag} (given by Haagerup and Harpe) is not always better than the inequalities discussed here and vice versa.

Finally, by using Corollary \ref{cor5} we obtain a reverse power inequality for the numerical radius.

\begin{cor}
	Let $T\in \mathcal{B}(\mathcal{H})$. If $\|T\|\leq 1$, then 
	\begin{eqnarray*}   %\label{rev}
		w^n(T) &\leq& \frac{1}{2^{n-1}} w(T^n)+  1- \frac{1}{2^{n-1}},
	\end{eqnarray*}
for every positive integer $n\geq 2.$
This inequality is sharp.
\end{cor}
\begin{proof}
	From Corollary \ref{cor5}, we have
	\begin{eqnarray*}
		w^n(T) &\leq& \frac{1}{2^{n-1}} w(T^n)+ \sum_{k=1}^{n-1} \frac{1}{2^{k}} \left\|T^k \right\| \left\|T \right\|^{n-k}\\
		&\leq& \frac{1}{2^{n-1}} w(T^n)+ \sum_{k=1}^{n-1} \frac{1}{2^{k}}  \left\|T \right\|^{n}\\
		&\leq& \frac{1}{2^{n-1}} w(T^n)+ \sum_{k=1}^{n-1} \frac{1}{2^{k}}  \\
		&=& \frac{1}{2^{n-1}} w(T^n)+  1- \frac{1}{2^{n-1}}.
	\end{eqnarray*}
If $T^*T=TT^*$ and $\|T\|=1,$ then 
\begin{eqnarray*}
	w^n(T) &=& \frac{1}{2^{n-1}} w(T^n)+  1- \frac{1}{2^{n-1}}=1.
\end{eqnarray*}
\end{proof}

\smallskip
\smallskip

%\noindent {\bf{Acknowledgments.}} The author would like to thank SERB, Govt. of India for the financial support in the form of National Post Doctoral Fellowship (NPDF, File No. PDF/2022/000325) under the mentorship of Professor Apoorva Khare.\\

\noindent {\bf{Data availability statements.}}\\
Data sharing not applicable to this article as no datasets were generated or analysed during the current study.\\

\noindent {\bf{Declarations.}}\\
\noindent {\bf{Conflict of Interest.}} The author declares that there is no conflict of interest.

\bibliographystyle{amsplain}

\end{document}